\documentclass[notitlepage,12pt]{article}

\usepackage{amsmath}
\usepackage{amsthm}
\usepackage{amssymb}
\usepackage[margin=2cm]{geometry}
\usepackage{url}

\theoremstyle{plain}
\newtheorem{Theorem}{Theorem}

\newtheorem{Corollary}[Theorem]{Corollary}
\newtheorem{Lemma}[Theorem]{Lemma}

\title{Nodal surfaces in $\mathbb{P}^3$ and coding theory}

\author{Sascha Kurz\\ {\small Dept.\ of Mathematics, University of Bayreuth, 95440 Bayreuth, Germany}\\ {\small sascha.kurz@uni-bayreuth.de}}

\begin{document}

\maketitle

\begin{abstract}
\noindent
To each nodal hypersurface one can associate a binary linear code. Here we show that the binary linear 
code associated to sextics in $\mathbb{P}^3$ with the maximum number of $65$ nodes, as e.g.\ the 
Barth sextic, is unique. We also state possible candidates for codes that might be associated with a 
hypothetical septic attaining the currently best known upper bound for the maximum number of nodes.

\smallskip


\end{abstract}


\section{Introduction}
\label{sec_introduction}
An irreducible hypersurface $S$ of degree $s$ in a complex projective space 
$\mathbb{P}^n:=\mathbb{P}^n(\mathbb{C})$ is the zero set of an irreducible homogeneous 
polynomial $f(x_0,\dots,x_n)$. A singularity is a point on the hypersurface 
where all partial derivatives vanish. A generic hypersurface is smooth i.e.\ it does not
contain any singularities, so that it is natural to ask which combinations of singularities 
can occur on a hypersurface of given degree $s$ in $\mathbb{P}^n$. For $s=1$ there cannot 
be any singularity and for $s=2$ there can be at most one isolated singularity which has 
to be an ordinary double point, or node for brevity, i.e.\ a singularity where the 
Hessian matrix is invertible. The possible combinations of singularities of cubics in $\mathbb{P}^3$
have been classified by Schl\"afli in 1863 \cite{schlafli1863x} and  the classification of all 
quartic surfaces in $\mathbb{P}^3$ with respect to their occurring singularities has been completed
in 1997 \cite{yang1997enumeration}. In $\mathbb{P}^2$ the maximum number of isolated singularities 
of a plane curve of degree $s$ is given by ${s\choose 2}$ and attained by $s$ lines in general position. In 
higher dimensions no such result is known yet. As the classification problem seems to be quite complex for 
larger parameters it makes sense to add further restrictions. A hypersurface is called 
nodal if all of its singularities are nodes. So, let $\mu(s)$ denote the maximum number of nodes of a 
nodal surface in $\mathbb{P}^3$ that can be described by a polynomial of degree $s$. Clearly we have 
$\mu(1)=0$ and $\mu(2)=1$. The values $\mu(3)=4$ and $\mu(4)=16$ are attained by the Cayley cubic \cite{cayley1869vii} and 
a Kummer surface \cite{kummer1864flachen}, respectively. In 1979 Beauville applied coding theory to determine $\mu(5)=31$ 
\cite{beauville1979nombre} and thereby improve a general upper bound of Basset \cite{basset1906maximum}.   
To this end let $\pi\colon \tilde{S} \rightarrow S$ be a minimal resolution 
of singularities. A set $N$ of nodes on $S$ is called even if there exists a divisor $Q$ on $\tilde{S}$ 
such that $2 Q\sim \pi^{-1}(N)$. The even sets of nodes on $S$ comprise the codewords of a binary linear 
code $C$, which we call the code associated to $S$. In \cite{jaffe1997sextic} Jaffe and Ruberman showed 
$\mu(6)<66$ by excluding the existence of certain binary linear codes, so that $\mu(6)=65$ due to the existence of the 
Barth sextic \cite{barth1996two}. As mentioned by Jaffe and Ruberman it would be very interesting to exactly describe 
the associated code for the known examples of nodal surfaces with many nodes. For quintic surfaces this has been 
accomplished by Beauville \cite{beauville1979nombre} and for generalized Kummer surfaces by Catanese \cite{catanese1996generalized}. 
While an explicit equation of the Barth sextic was known its associated code was first determined by showing the 
uniqueness of the possible codes of sextics with $65$ nodes \cite{kurz2020classification}, cf.\ \cite{catanese2022varieties} 
for a more direct derivation. Due to the importance of the problem, since currently none of the two mentioned preprints 
is published, and in order to popularize the problem we would like to give a streamlined proof of the uniqueness of the 
code associated to a nodal sextic with $65$ nodes in this paper. Additionally we give some information on the next open case of 
degree $d=7$.

The remaining part of this paper is structured as follows. In Section~\ref{sec_introduction} we introduce the necessary preliminaries 
from algebraic geometry and coding theory. The uniqueness of the code associated to a sextic in $\mathbb{P}^3$ with the maximum number 
of nodes is then concluded in Section~\ref{sec_uniqueness}. We close with a conclusion and a few remarks on further open problems, 
including the maximum number of nodes of septics in $\mathbb{P}^3$, in Section~\ref{sec_conclusion}.

\section{Preliminaries}
\label{sec_preliminaries}
An $[n,k]_2$ code $C$, or binary linear code with length $n$ and dimension $k$, is a $k$-dimensional subspace of $\mathbb{F}_2^n$. 
The elements of $C$ are called codewords. The dual code $C^\perp$ of $C$ is the subspace that is perpendicular to $C$, i.e.\ $C^\perp$ is 
an $[n,n-k]_2$ code. The weight $\operatorname{wt}(c)$ of a codeword $c\in C$ is the number of non-zero entries in 
$c$ and the minimum weight $d$ of $C$ is the minimum weight over of the non-zero codewords in $C$, so that we also speak of an $[n,k,d]_2$ code. If 
the minimum weight $d^\perp$ of the dual code $C^\perp$ is at least $3$ we say that $C$ is projective. If 
the non-zero weights of codewords in $C$ are contained in $\left\{w_1,\dots,w_l\right\}$, then we say that $C$ is a $\Delta$-divisible 
$[n,k,\{w_1,\dots,w_l\}]_2$ code if $w_i$ is divisible by $\Delta$ for all $1\le i\le l$. The polynomial $W_C(x):=\sum_{c\in C} x^{\operatorname{wt}(c)}y^{n-\operatorname{wt}(c)}$ is 
called the weight enumerator of $C$. The MacWilliams identity determines the weigh enumerator of the dual code via $W_{C^\perp}(x,y)=\tfrac{1}{2^k}\cdot W_C(x+y,x-y)$ 
\cite{macwilliams1962combinatorial,macwilliams1963theorem}. 
Each $k\times n$ matrix over $\mathbb{F}_2$ whose row span equals $C$ is called a generator matrix of $C$. If 
$C$ contains a codeword with a non-zero entry in position $i$ for all $1\le i\le n$, which is equivalent to $d^\perp\ge 2$, then $C$ has full length and $n$ is called the effective length 
of $C$. In order to ease the notation we only consider $[n,k]_2$ codes of full length. For a given codeword $c$ of an $[n,k]_2$ code the support 
$\operatorname{supp}(c)$ is given by $\{1\le i\le n\,:\, c_i\neq 0\}$, i.e., the cardinality of its support equals its weight. With this the residual 
code $\operatorname{Res}(C;c)$ of an $[n,k]_2$ code $C$ with respect to a non-zero codeword $c$ is the restriction of the codewords to 
$\{1,\dots,n\}\backslash \operatorname{supp}(c)$, which has effective length $n-\operatorname{wt}(c)$.

The code associated to the Cayley cubic is an $[4,1,\{4\}]_2$ code with weight enumerator $W_C(x,y)=x^0y^4+x^4y^0$ and the code associated with a Kummer surface 
is an $[16,5,\{8,16\}]_2$ code with weight enumerator $W_C(x,y)=x^0y^{16}+30x^8y^8+x^{16}y^0$. Geometrically the latter code corresponds to the $16$ points of an affine 
solid and a generator matrix is given by
$$
  \left(\begin{smallmatrix}
  0000000011111111\\
  0000111100001111\\
  0011001100110011\\
  0101010101010101\\
  1111111111111111
  \end{smallmatrix}\right).
$$ 
The code associated to a quintic with $31$ nodes is given by a $[31,5,\{16\}]_2$ code with weight enumerator $W_C(x,y)=x^0y^{31}+31x^{16}y^{15}$. Geometrically such a 
code corresponds to the $31$ points of the projective geometry $\operatorname{PG}(4,2)$ and a generator matrix is given by
$$
  \left(\begin{smallmatrix}
  0000000111111110000000011111111\\
  0001111000011110000111100001111\\
  0110011001100110011001100110011\\
  1010101010101010101010101010101\\
  0000000000000001111111111111111
  \end{smallmatrix}\right).
$$
The following general properties of the associated $[n,k]_2$ code $C$ of a nodal surface with degree $s$ and $m$ nodes are known. For the length we have $n\le m$ 
and for the dimension a general argument of Beauville \cite{beauville1979nombre} gives $k\ge m-\left\lceil s^3/2\right\rceil +2s^2-3s+1$, see 
\cite[Proposition 4.3]{jaffe1997sextic}. If $d$ is odd, then $C$ is $4$-divisible and $8$-divisible otherwise, see \cite[Proposition 2.11]{catanese1981babbage}. 
The minimum distance $d$ satisfies $d\ge 2\lceil s(s-2)/2\rceil$, see \cite[Theorem 1.10]{endrass1997minimal}. So, for sextics with $65$ nodes the associated code 
is a $8$-divisible $[n,k,24]_2$ code with $n\le 65$ and $k\ge 12$.

\section{Uniqueness of the code associated to a sextic with $\mathbf{65}$ nodes}
\label{sec_uniqueness}

First we want to conclude some further properties of $8$-divisible $[n,12,24]_2$ codes (of full length) by purely coding theoretic arguments. First we observe that 
for a $q^r$-divisible linear code over $\mathbb{F}_q$ with $r\ge 1$ each residual code is $q^{r-1}$-divisible, see e.g.\ \cite[Lemma 7]{honold2018partial}. Note 
that we have $C\subseteq C^\perp$ for each $4$-divisible $[n,k]_2$ code, so that $k\le n/2$. Setting $W_C(x,y)=\sum_i a_ix^iy^{n-i}$ and $W_{C^\perp}(x,y)=\sum_i a_i^*x^iy^{n-i}$ 
the equations for the coefficients $y^0$, $y^1$, $y^2$, and $y^3$ in the MacWilliams identity can be rewritten to 
\begin{eqnarray}
  \sum_{i>0} a_i &=& 2^k-1,\\
  \sum_{i\ge 0} ia_i &=& 2^{k-1}n,\\
  \sum_{i\ge 0} i^2a_i &=& 2^{k-1}(a_2^*+n(n+1)/2),\\ 
  \sum_{i\ge 0} i^3a_i &=& 2^{k-2}(3(a_2^*n-a_3^*)+n^2(n+3)/2).
\end{eqnarray} 
We also speak of the first four MacWilliams identities. In this special form, those equations are also known as the 
first four (Pless) power moments \cite{pless1963power}.

\begin{Lemma}
  \label{lemma_n_ge_63}
  Let $C$ be a binary $8$-divisible linear code with minimum distance $d\ge 24$, dimension $k=12$ and effective length $n\le 65$, then
  $a_{40}\ge 1$ and $n\ge 63$.
\end{Lemma}
\begin{proof}
  Solving the first four MacWilliams identities for $a_{24}$, $a_{32}$, $a_{40}$, and $a_{48}$ gives
  $$
    a_{40}=\frac{205}{2}n^2-6808n-\frac{1}{2}n^3+(208-3n)a_2^*+3a_3^*+6a_{56}+20a_{64}+147420 
  $$
  and
  $$
    a_{40}+a_{48}=71n^2-\frac{14504}{3}n-\frac{1}{3}n^3+(144-2n)a_2^*+2a_3^*+2a_{56}+10a_{64}+106470.
  $$    
  Since $a_2^*,a_3^*,a_{56},a_{64}\ge 0$, $208-3n\ge 0$, $144-2n\ge 0$ we have 
  $$
    a_{40}\ge \frac{205}{2}n^2-6808n-\frac{1}{2}n^3+147420  
  $$
  and   
  $$
    a_{40}+a_{48}\ge 71n^2-\frac{14504}{3}n-\frac{1}{3}n^3+106470.  
  $$  
  For $54\le n\le 60$ we have $a_{40}+a_{48}<0$, which is impossible. If either $n\le 53$ or $61\le n\le 65$, 
  then $a_{40}\ge 1$. Thus, $a_{40}\ge 1$. Consider the residual code $C'$ of a codeword of weight $40$. $C'$ has dimension $11$ 
  and is $4$-divisible, so that its length is at least $23$, see e.g.\ \cite[Section VIII]{gaborit1996mass}.  
\end{proof}

We remark that there even exists an $8$-divisible $[64,13,24]_2$ code $C$
with generator matrix 
$$
\left(\begin{smallmatrix}
1111111111111111111111111111111111111111111111111111111111111111\\
0000000000000000000000000000000011111111111111111111111111111111\\
0000000000000000111111111111111100000000000000001111111111111111\\
0000000011111111000000001111111100000000111111110000000011111111\\
0000111100001111000011110000111100001111000011110000111100001111\\
0011001100110011001100110011001100110011001100110011001100110011\\
0101010101010101010101010101010101010101010101010101010101010101\\
0000000000001111001100110011110000111100001100111111000011111111\\
0001001000011101000100100001110101001000010001110100100001000111\\
0000011001100000001110101010001101011100110001011001111111111001\\
0000000001101001000000001001011001011010110011000101101000110011\\
0001001000011101011110110111010000101110110111101011100001001000\\
0000001101010110011001011100111100000011101010011001101011001111\\
\end{smallmatrix}\right).
$$
With these properties the code is unique, as shown by exhaustive enumeration in \cite{kurz2020classification}.  
Its automorphism group has order $23224320$ and its weight enumerator 
is given by $W_C(x,y)=x^{0}y^{64}+1008x^{24}y^{40}+6174x^{32}y^{32}+1008x^{40}y^{24}+x^{64}y^{0}$. 
The code was obtained in \cite{delsarte1975alternating} and has the following nice description, see \cite{jaffe1997sextic}:
It is a subcode of the second order Reed-Muller code $R(2,6)$ containing the first order Reed-Muller code $R(1,6)$ as a subcode. The
cosets of $R(1,6)$ in it correspond to the symplectic forms $B_a$ in $\mathbb{F}_{64}$, given by $B_a(x,y)=\operatorname{tr}((ax^4 + a^{16}x^{16})y)$.
So, subcodes of this code give examples of $8$-divisible $[n,12,24]_2$ codes for $n=63$ and $n=64$. We remark that an $[59,12,24]_2$ code exists 
while the existence of a $[58,12,24]_2$ code is currently unknown. So the $8$-divisibility of the code is a severe restriction.

\begin{Lemma}
  Let $C$ be an $8$-divisible $[65,12,24]_2$ code. Then $C$ does not contain a codeword of weight $64$ or $56$.
\end{Lemma}
\begin{proof}
  The residual code of a codeword of weight $64$ would be a $4$-divisible code of length $1$, which obviously cannot exist. Similarly, the residual 
  code of a codeword of weight $56$ would be a $4$-divisible code of length $9$. However, no such code exists \cite{kiermaier2020lengths}.
\end{proof}

We remark that it is also possible to exclude the existence of a codeword of weight $48$ in an $8$-divisible $[65,12,24]_2$ code $C$ by theoretical arguments. 
For a codeword $c$ of weight $48$ the residual code $\operatorname{Res}(C;c)$ is a $[17,k,\{4,8,12\}]_2$ code, so that $k\le 8$ \cite[Section VIII]{gaborit1996mass}. The restriction of the 
codewords of $C$ to the coordinates in the support of $c$ gives a $[48,k',\{24,48\}]_2$ code $C'$. Codes whose only occurring non-zero weights are the minimum distance $d$
and $2d$ are completely characterized in \cite{jungnickel2018classification}. In our situation $C'$ is a uniquely defined $[48,k',\{24,48\}]_2$ code with $k'\in\{4,5\}$, 
so that $k\in\{7,8\}$. Up to isomorphism there are only five $[17,7,\{4,8,12\}]_2$ codes and a unique $[17,8,\{4,8,12\}]_2$ code \cite{doran2011codes,MillerWP}. For a 
computer-free classification for the possibilities for $\operatorname{Res}(C;c)$ one may consider the decomposition of $\operatorname{Res}(C;c)$ into subcodes spanned by codewords of weight $4$, which is
completely characterized in a more general setting \cite{kiermaier2020classification}. Having the explicit possible choices for $C'$ and $\operatorname{Res}(C;c)$ at hand one can easily 
exclude the existence of $C$ with a codeword $c$ of weight $48$. However, since those arguments are quite lengthy, due to many case differentiations, and such a result 
can be easily obtained by exhaustive computer enumeration within seconds we do not go into details.    

\begin{Lemma}
  \label{lemma_ten_cases}
  Let $C$ be an $8$-divisible $[n,12,24]_2$ code with $n\le 65$. Then the non-zero weights of $C$ are contained in $\{24,32,40,64\}$.
\end{Lemma}
\begin{proof}
  From Lemma~\ref{lemma_n_ge_63} we know that $n\ge 63$ and that $C$ contains a codeword $c$ of weight $40$. The residual code then is a $4$-divisible $[n-40,11]_2$ code.
  Using the software package \texttt{LinCode} \cite{bouyukliev2020computer} we have exhaustively enumerated all possibilities for $\operatorname{Res}(C;c)$ up to isomorphism.
  There are exactly $11$ $[23,11]_2$, $83$ $[24,11]_2$, and $215$ $[25,11]_2$ $4$-divisible codes cf.\ \cite{doran2011codes,MillerWP}. Starting from these residual codes 
  we have exhaustively enumerated all possibilities for $C$ using \texttt{LinCode}. Up to isomorphism there are unique $[n,12,24]_2$ codes for $n\in\{63,65\}$ and eight 
  $[64,12,24]_2$ codes (assuming $8$-divisibility) and in none of these cases codewords of weight $48$ or $56$ occur. (Generator matrices for these ten cases are given 
  in \cite{kurz2020classification}.) 
\end{proof}

We remark that the initial classification of the $8$-divisible $[\le 65,12,24]_2$ codes in \cite{kurz2020classification} took almost 1000 single-core CPU hours on a computing cluster. 
There are e.g.\ $978528$ $8$-divisible $[63,10,24]_2$ and $704571$ $8$-divisible $[64,11,24]_2$ codes. Constructing the $8$-divisible $[\le 65,12,24]_2$ codes via the residual 
codes of a codeword of weight $40$ allows to keep the numbers of intermediate codes in the enumeration process much smaller. The determination problem of possible 
codes associated with a sextic with $65$ nodes was indeed the initial motivation to develop a new enumeration algorithm for linear codes. Over the time there were algorithmic 
improvements, see \cite{kurz2020classification,kurz2024computer} for details, so that the computation underlying Lemma~\ref{lemma_ten_cases} can now be performed in less than 
two hours on a single core. 

In \cite{jaffe1997sextic} it was shown that the code associated to a sextic with $65$ nodes cannot contain a codeword with weight $48$ or $64$. This boils down the number of possible codes
to three: 
\begin{Lemma}
  \label{lemma_three_cases}
  Let $C$ be an $[n,12,\{24,32,40,48,56\}]_2$ code with $n\le 65$. Then $C$ is isomorphic to one of the three cases:
    \begin{enumerate}
    \item[(1)]
    $[63,12,24]_2$\\
    $\left(\begin{smallmatrix} 
    001100001110000001111101000011111110010010100100001100000000000\\
    101001111111000000110111010000100110100011011000000010000000000\\
    000100111011100011110111001000010000110000110110100001000000000\\
    010001111111110011001100001001100100010001101000001000100000000\\
    110001110000010111001111011000011100100011000010100000010000000\\
    000000011000110111100011010011101110010001011110000000001000000\\
    010011110001111101010000110100100011101110111111111000000100000\\
    001000110111101100001111110000000001100110000111100000000010000\\
    000111110001100011000000001100011111100001100001111000000001000\\
    000000001111100000111111111100000000011111100000011000000000100\\
    000000000000011111111111111100000000000000011111111000000000010\\
    000000000000000000000000000011111111111111111111111000000000001\\
    \end{smallmatrix}\right)$\\
    $W_C(x,y)=x^{0}y^{63}+630x^{24}y^{39}+3087x^{32}y^{31}+378x^{40}y^{23}$\\
    $\# \operatorname{Aut}(C)=362880$
    \item[(2)]
    $[64,12,24]_2$\\
    $\left(\begin{smallmatrix}
    0000110001101110000100100100100011011000011011011110100000000000\\
    1011110000100110010000001100010000111101001110111000010000000000\\
    1010110001001010110010000000101111110000001100101011001000000000\\
    1111100000001100000010100100111101000011011011101000000100000000\\
    0111000000001010110110001100011000000110111100110011000010000000\\
    0000000100001001111110011010010101001101010101010101000001000000\\
    0101011111010000010001111001110011000100100000101100000000100000\\
    0011010011001000001111111001111111011111010011011100000000010000\\
    0000101111000110000000000111101111000011001111000011000000001000\\
    0000011111000001111111111111100000111111000000111111000000000100\\
    0000000000111111111111111111100000000000111111111111000000000010\\
    0000000000000000000000000000011111111111111111111111000000000001\\
    \end{smallmatrix}\right)$\\
    $W_C(x,y)=x^{0}y^{64}+502x^{24}y^{40}+3087c^{32}y^{32}+506x^{40}y^{24}$\\
    $\# \operatorname{Aut}(C)=5760$
    \item[(3)]
    $[65,12,24]_2$\\
    $\left(\begin{smallmatrix}  
    10000100000000110110010001110100111101010001011110010100000000000\\
    10100100011000001001000110100110111111001000001100011010000000000\\
    01000010011100011000000100110100110000011111011110001001000000000\\
    11110100001110110100000011010110100001011100000100001000100000000\\
    01101011000001100011010001000011001010001111000010111000010000000\\
    00101001110111101011000001011000000110111001001000100000001000000\\
    00011000111111100000111110001000100010001010101001100000000100000\\
    00000111001011100111110001010100000001111001100000011000000010000\\
    00011111000111100000001111001101111110000111100111111000000001000\\
    00000000111111100000000000111100011110000000011111111000000000100\\
    00000000000000011111111111111100000001111111111111111000000000010\\
    00000000000000000000000000000011111111111111111111111000000000001\\
    \end{smallmatrix}\right)$\\
    $W_C(x,y)=x^{0}y^{65}+390x^{24}y^{41}+3055x^{32}y^{33}+650x^{40}y^{25}$\\
    $\# \operatorname{Aut}(C)=15600$
  \end{enumerate}
  All three codes are projective.
\end{Lemma}

\begin{Theorem}
  \label{thm_main}
  If $S$ is a sextic with $65$ nodes in $\mathbb{P}^3$ then its associated code is given by Lemma~\ref{lemma_three_cases}.(3).
\end{Theorem}
\begin{proof}
  Let $C$ be the code associated to a sextic in $\mathbb{P}^3$. In \cite{endrass1997minimal} the existence of a code $C'\supsetneq C$ with $\dim(C')=\dim(C)+1$ 
  such that the codewords in $C'\backslash C$ have weights contained in $\{16,28,32,36,\dots\}$ was shown. First we use \texttt{LinCode} to check that none 
  of the cases of Lemma~\ref{lemma_three_cases} can be extended to a $[\le 66,13,\{24,32,40,48,56\}]_2$ code, so that there are just three choices for $C$. Then 
  we used \texttt{LinCode} to extend $C$ to $C'$.   
\end{proof}
We remark that the code $C'$ in the proof of Theorem~\ref{thm_main} is unique. A generator matrix its given by
$$
  \left(\begin{smallmatrix}
  100000000011000100000110100000010001100011111110011111001111001010\\
  010000001000101000001110000010111100000110011011000100101100010100\\
  001000001011111100001011110000101110100110011000100000110110110111\\
  000100001001000100000111100001010001101010101010011000001010001011\\
  000010001010010000001101000101100000111001010100010101000101110010\\
  000001001001111000000101011000111011000100111111010111011001011100\\
  000000101011001100001100100010110101100111111111011110000011001100\\
  000000011010110100000110101100100010101100111110011111001111001010\\
  000000000111100000001111101001010011110011111110010110010101011010\\
  000000000000000010001000001000100100010011111111101110101010100101\\
  000000000000000001000100000100011000100011111111011101101001011010\\
  000000000000000000100010100010000001000111111110111011010110101010\\
  000000000000000000010001010001000010001011111101110111010101010101
  \end{smallmatrix}\right).
$$ 
The automorphism group of $C'$ has order $15600$ and the weight enumerator of $C'\backslash C$ is given by $26x^{16}y^{50} 
+ 650x^{28}y^{38} + 1690x^{32}y^{34} + 1300x^{36}y^{30} + 300x^{40}y^{26} + 130x^{44}y^{22}$. For a nice description of 
the associated code $C$ of a sextic with $65$ nodes and its automorphism group we refer to \cite[Section 4]{catanese2022varieties}. Since the residual codes 
of codewords of weight $40$ played an important role in the determination of $C$ we mention that they all are unique with weight enumerator 
$W_{\operatorname{Res}(C;c)}(x,y)=x^{0}y^{25}+3x^{4}y^{21}+258x^{8}y^{17}+1278x^{12}y^{13}+493x^{16}y^9+15x^{20}y^5$ and an automorphism group of order $4608$. 
A generator matrix is given by 
$$
  \left(\begin{smallmatrix}  
    0000100100100100000000000\\ 
    0100011010100111000000000\\ 
    1010010000101110100000000\\ 
    1100100010101010010000000\\ 
    1111101111001010001000000\\ 
    0101111000010100000100000\\ 
    0011101100010010000010000\\ 
    0011100001101010000001000\\ 
    0000100000000110000000100\\ 
    0000011100011110000000010\\ 
    0000000011111110000000001\\ 
  \end{smallmatrix}\right).
$$

While Theorem~\ref{thm_main} shows the uniqueness of the code associated to a sextic with the maximum number of nodes, the sextic itself is far from being unique. In 
\cite[Theorem 5.5.9]{PhdPettersen} a $3$-parameter family of sextics with $65$ nodes, including the Barth sextic, was obtained; cf.\ \cite[Theorem 196]{catanese2022varieties}.

The non-zero weights occurring in a code associated to a nodal sextic in $\mathbb{P}^3$ are contained in $\{24,32,40,56\}$ and all cases can indeed occur \cite{catanese2007even,van2018genus}.
The exclusion of weight $56$ in \cite{casnati1997even} turned out to be incorrect \cite{casnati1998errors}. In \cite[Lemma 2.3]{pignatelli2007wahl} it was show that a binary code 
with non-zero weights in $\{24,32,56\}$ has dimension at most $10$ and in \cite[Lemma 2.1]{pignatelli2007wahl} it was show that a binary code with non-zero weights in $\{24,32\}$ has dimension 
at most $9$.
\begin{Lemma}
  Let $C$ be a $[n,k,\{24,32\}]_2$ code with $n\le 56$. Then, we have $k\le 8$. Moreover, if $k=8$, then $n\in\{51,54,55,56\}$. For $(n,k)=(51,8)$ the code is projective.
\end{Lemma} 
\begin{proof}
  Since each $[n,8,24]_2$ code satisfies $n\ge 51$, see \cite{bouyukhev2000smallest}, we have $51\le n\le 56$. The fact that each $[n,7,24]_2$ code satisfies $n\ge 50$ implies 
  that the code is projective for $(n,k)=(51,8)$. For an $[n,8,\{24,32\}]_2$ code we apply the dual transform as in \cite[Theorem 3.3]{bouyukliev2024dual} with $\alpha=\tfrac{1}{8}$ and 
  $\beta=-3$. The projective dual code $D_{\alpha,\beta}(C)$ has length $128\alpha n+255\beta$ and minimum distance $64(\alpha n+2\beta)$ if $n<2^k-1$.  
  Since no $[67,8,32]_2$ code and no $[83,8,40]_2$ code exists, see \cite{bouyukhev2000smallest}, we have $n\notin\{52,53\}$. Let $C'$ be the $[56,9,\{24,32,56\}]_2$ code arising from 
  $C$ by adding a codeword of weight $56$. Then we clearly have $a_{56}=1$ and $a_{24}=a_{32}$ for $C'$. Using the MacWilliams identities we compute $a_{24}=a_{32}=255$ and $a_2^*=10$.
  Let $\mathcal{P}'$ be the multiset of points of cardinality $56$ in $\operatorname{PG}(8,2)$ corresponding to $C'$, see e.g.\ \cite{dodunekov1998codes}. Projection of $\mathcal{P}'$ 
  through a point of multiplicity $m\ge 1$ yields a multiset of points $\mathcal{P}$ of cardinality $56-m$ in $\operatorname{PG}(7,2)$ that corresponds to an $[56-m,8,\{24,32\}]_2$ code, 
  so that $m\in\{1,2,5\}$. Since each point of multiplicity $m$ contributes ${m\choose 2}$ to $a_2^*$ the multiset $\mathcal{P}'$ consists either of one point of multiplicity $5$ and $51$ 
  points of multiplicity $1$ or five points of multiplicity $2$ and $46$ points of multiplicity $1$.
  
  Now assume that $H$ is an $[l,9,\{24,32\}]_2$ code and $H'$ the $[56,10,\{24,32,56\}]_2$ code that arises from $H$ by adding a codeword of weight $56$. Let $\mathcal{Q}$ and 
  $\mathcal{Q}'$ be the multisets of points corresponding to $H$ and $H'$, respectively. For $H'$ we compute $a_{24}=a_{32}=511$, $a_{56}=1$, $a_2^*=7$, and $a_3^*=0$. Clearly, 
  $\mathcal{Q}'$ contains a point of multiplicity $1$ and projection through this point yields a multiset of points that corresponds to a $[55,9,\{24,32\}]_2$ code. W.l.o.g.\ we 
  assume that $H$ was chosen as such a code, i.e., we have $l=55$. From the classification of the possible lengths of $[\le 56,8,\{24,32\}]_2$ codes we conclude that the 
  point multiplicities in $\mathcal{Q}$ are contained in $\{0,1,4\}$. From the MacWilliams identities we compute $a_{24}=284$, $a_{32}=227$, and $a_2^*=7$ for $H$. 
  Since each point of multiplicity $m$ contributes ${m\choose 2}$ to $a_2^*$ this is impossible. 
\end{proof}
Up to isomorphism there are unique $[51,8,\{24,32\}]_2$ and $[54,8,\{24,32\}]_2$ codes, two $[55,8,\{24,32\}]_2$ codes, three two $[56,8,\{24,32\}]_2$ codes, and 
two $[56,9,\{24,32,56\}]_2$ codes. Generator matrices are given by
\begin{eqnarray*}
  \!\!\!\!\!\!&&\left(\begin{smallmatrix}
    101100011111111111100001010111111111100111110000000\\
    010011011111111000011110101111111110011111101000000\\
    000000100001111111111110000000011111111001100100000\\
    001111100110011001100110011001100110101010100010000\\
    001111111000011110011001100001111000101100100001000\\
    110111100110000011101011000010111010010101100000100\\
    110101101010101010011100011100001011100110100000010\\
    000110100011001100101101111111010101010100100000001  
  \end{smallmatrix}\right),
  \left(\begin{smallmatrix}
    111111111111000000111111111111110000000001111110000000\\
    000000000000111111111111111111110000000100011001000000\\
    000000011111000000000000001111111111111001111100100000\\
    000111111111111111000011111111110001111010000100010000\\
    001000111111001111001100110000110010001010111000001000\\
    001111000111000001000101110011000110011111100100000100\\
    011111111011000110000110001101001110010010001000000010\\
    101001111001110000000110110001001001101011101100000001
  \end{smallmatrix}\right),\\
  \!\!\!\!\!\!&&\left(\begin{smallmatrix}
    1111111111111000000111111111111110000000001111010000000\\
    0000000000000111111111111111111110000000100011001000000\\
    0000000111111000000000000001111111111111001111000100000\\
    0001111111111111111000011111111110001111010000000010000\\
    0110011001111000011000011110000110110011001100100001000\\
    0011101010001001101011101111111001011111111111000000100\\
    1110001101110011111000110110111010111100111111000000010\\
    1111101010010101100100011010001100001100110011000000001
  \end{smallmatrix}\right),
  \left(\begin{smallmatrix}
    1111111111111111111000000000000000111111111111010000000\\
    1111111110000000000111111111100000111111111110101000000\\
    1111111110000111111000001111111111000000111110100100000\\
    0000000000000000111001111111100111000111111111000010000\\
    0000011110011111011000010001111001111000000011100001000\\
    1111100010100011111010110110100010001011000100000000100\\
    1111100101001101001010000111101000011100001001100000010\\
    0000001110110011010111110010010000001101011111000000001
  \end{smallmatrix}\right),\\
  \!\!\!\!\!\!&&\left(\begin{smallmatrix}
    11111111111111111111111111000000000000111110000010000000\\
    00000011111111111111000000111111000000001101000001000000\\
    11111100000000111111000000000000111111111100100000100000\\
    11111100001111111111001111111111001111000000010000010000\\
    00111100000011000011111111000011110011001100001000001000\\
    11001100111100001100110011001100001111110000000100000100\\
    00011111001100110100000000010101001111010101011100000010\\
    11001101011111001011000100000001010111000111000100000001  
  \end{smallmatrix}\right),
  \left(\begin{smallmatrix}
    11111111111111111111111111111110000000000000000010000000\\
    00000000111111111111111100000001111000011100000001000000\\
    11111111000000001111111100000000000111100011100000100000\\
    00001111000011110000111100000001111111100000011100010000\\
    00111111011100010001000100011110011011100000100100001000\\
    00010001100100110001011100000110100111111101101000000100\\
    11110011101111010111100000011110000101101111111100000010\\
    01011111111000010001011101100001001010001101010000000001
  \end{smallmatrix}\right),\\
  \!\!\!\!\!\!&&\left(\begin{smallmatrix}
    11111111111111111111111111111110000000000000000010000000\\
    00000000111111111111111100000001111000011100000001000000\\
    11111111000000001111111100000000000111100011100000100000\\
    00001111000011110000111100000001111111100000011100010000\\
    00010111001100110001000101111110001001100100011100001000\\
    11111001011100010001011110111110110011101111100100000100\\
    01100000110101000010111010111110011010011000101000000010\\
    10101010000111010111101110111111111000110101011100000001
  \end{smallmatrix}\right),
  \left(\begin{smallmatrix}
    11111111111111111111111111000000000000111110000100000000\\
    00000011111111111111000000111111000000001101000010000000\\
    11111100000000111111000000000000111111111100100001000000\\
    11111100001111111111001111111111001111000000010000100000\\
    00111100000011000011111111000011110011001100001000010000\\
    00000101110100001100011111000111000001011111101000001000\\
    00001100011011010011100000011101111101000110110000000100\\
    11000001110100011111100011001001001110001010110000000010\\
    11110111101000111100101100101111001110101111101000000001  
  \end{smallmatrix}\right),\\
  \!\!\!\!\!\!&&\left(\begin{smallmatrix}
    11111111111111111110000001111100000000001111111100000000\\
    11111111111111000001111110000011111000001111101010000000\\
    00000000000000111111111110000000000111111111111001000000\\
    00000001111111111110001110011111111001111111100000100000\\
    00111110001111011111110110011100111110010011110000010000\\
    01000000010111101110010001100101011110010101110000001000\\
    11000000111011110111100010101100100010110010010000000100\\
    00111110010001000010110101100100001101011110101000000010\\
    01111110110010000101000000101110110110010101001000000001  
  \end{smallmatrix}\right).
\end{eqnarray*}  

\begin{Corollary}
  Let $C$ be a $[56,k,\{24,32,56\}]_2$ code. Then, we have $k\le 9$.
\end{Corollary}

We remark that there exists a $[56,10,24]_2$ code with weight enumerator $W_C(x,y)=x^0y^{56}+399x^{24}y^{32}+224x^{28}y^{28}+399x^{32}y^{24}+x^{56}y^{0}$.

\section{Conclusion and outlook}
\label{sec_conclusion}
We have determined the unique code associated to a sextic in $\mathbb{P}^3$ with the maximum number of nodes in Theorem~\ref{thm_main}. For each $0\le \mu\le 65$ there exists a sextic 
in $\mathbb{P}^3$ with $\mu$ nodes \cite{barth1996two,catanese1982constructing}. Using similar techniques for sextics with $64$ codes there currently remain just seven possible candidates 
for the associated code \cite[Theorem 227]{catanese2022varieties}. However for sextics with $63$ nodes such a list of candidates would be rather huge, so that there is a need for more 
necessary conditions on the associated codes that might by exploited using coding theory techniques. 

For septics in $\mathbb{P}^3$ the currently best known bounds on the maximum number of nodes are given by $99\le \mu(7)\le 104$ \cite{PhdLabs,labs2004septic,varchenko1983semicontinuity}. 
With respect to the upper bound we mention that the $[96,10,44]_2$ code with weight enumerator $W_C(x,y)=x^0y^{96}+504x^{44}y^{52}+124x^{48}y^{48}+336x^{52}y^{44}+56x^{60}y^{36}+3x^{64}y^{32}$ 
given by the generator matrix
$$
  \left(\begin{smallmatrix}
    100000000010000100111101100000110001001101111110001000001101010100110011101101110011111000101100\\
    010000000001100010011110110000011000000110111111100100000110101010011001110110111001011100010110\\
    001000000010010101110010111000111101101110100001011010001110100001111111010110101111110110100111\\
    000100000001001010111001111100011110110111010000001101000111110000111111101011010111111011010011\\
    000010000010100001100001111110111110010110010110101110101110101100101100011000011000000101000101\\
    000001000001110000110000011111011111001011001011010111010111010110010110001100001100100010100010\\
    000000100010011100100101001111011110101000011011100011100110011111111000101011110101101001111101\\
    000000010001101110010010000111101111110100001101010001110011001111111100110101111010110100111110\\
    000000001010010011110100000011000110110111111000100000110100010011001101110111001110100010110011\\
    000000000101001001111010000001100011011011111100010000011010101001100110011011100111110001011001
  \end{smallmatrix}\right)
$$
satisfies all constraints on the weights from \cite{endrass1997minimal}. In the special case where an even set of nodes 
is cut out by a smooth cubic surface the corresponding codeword must have weight $36$, while weight $40$ cannot occur \cite{endrass1997minimal}. To this end,
a $[94,10,36]_2$ code with weight enumerator $W_C(x,y)=x^0y^{94}+120x^{36}y^{58}+182x^{44}y^{50}+489x^{48}y^{46}+192x^{52}y^{42}+14x^{56}y^{38}+26x^{60}y^{34}$ 
is given by the generator matrix
$$
  \left(\begin{smallmatrix}
    1111111111111110000000000001111110000001111111111110000001100000000000000000000000100000000000\\
    0000001111110001111110000000000001111111111111111110001110011000000000000000000000010000000000\\
    0001110001110000000001111111111111111110000001111110001110000110000000000000000000001000000000\\
    1111111111110001111111111110000001111110000000000001110000000001100000000000000000000100000000\\
    0001111110001110001110001110001110001110001110001111111110000000011000000000000000000010000000\\
    0001110000001110000001110001111111110001111110000001111110000000000110011100000000000001000000\\
    0001110001111111111110001111110001110000001110001110000000000000000001111100000000000000100000\\
    0010010010010110110110010110010010110010110010010010010110111110101110101111111000000000010000\\
    1111111110000000001111110000001110000000001111110001110000000000000000011100011111000000001000\\
    0000110100100010011010111010100101010001010110110010110011101011010011100101101011000000000111
  \end{smallmatrix}\right).
$$
So, additional restrictions on the code are needed in order to improve upon the upper bound for $\mu(7)$ using coding theoretic methods. The algorithmic tools for 
the exhaustive enumeration of linear codes are available. A first step might be the computation of associated codes of known hypersurfaces with many nodes.

For another type of singularities, so-called cusps, one can associate a $3$-divisible code over $\mathbb{F}_3$ \cite{barth2007cusps}, see also 
\cite{PhdLabs} for a general overview on hypersurfaces with many singularities.

On the coding theory side there are a few more sophisticated techniques that we have not applied in Section~\ref{sec_uniqueness}. If a binary linear code is 
$2^r$-divisible then there are some restrictions on the number of codewords whose weights are divisible by $2^{r+1}$, see e.g.\ \cite[Proposition 5]{dodunekov1999some}. 
The MacWilliams identities can be generalized to coordinate partitions \cite{simonis1995macwilliams}. In combination with linear programming this was used quite successfully 
to show the non-existence of some binary linear codes or to even show uniqueness of some optimal codes, see e.g.\ \cite{jaffe2000optimal}. However, applying those techniques 
usually comes with extensive computer calculations. So, in order to keep the paper relatively short and since all of the used computer calculations were performed in 
a reasonable short amount of time we have not applied those techniques to our problem.

\newcommand{\etalchar}[1]{$^{#1}$}


\end{document}